\documentclass{amsart}
\usepackage{amssymb,amsmath,latexsym}
\usepackage{amsthm}
\usepackage{fontenc}
\usepackage{amssymb}
\numberwithin{equation}{section}

\newtheorem{theorem}{Theorem}[section]

\begin{document}
\author{Alexander E Patkowski}
\title{A partition identity connected to the second rank moment}

\maketitle
\begin{abstract} The purpose of this note is to offer some partition implications of a $q$-series that is connected to the second Atkin-Garvan moment. Inequalities and relationships among the number of divisors and partitions are provided as consequences. \end{abstract}

\keywords{\it Keywords: \rm Bailey pairs; Partitions; $q$-series.}

\subjclass{ \it 2010 Mathematics Subject Classification 11P84, 11P81}

\section{Introduction and main theorems}

We use the familiar basic-hypergeometric notation [6] $(a)_n=(a;q)_{n}:=\prod_{0\le k\le n-1}(1-aq^{k})$ to display our $q$-series in compact form. It is now well-known in the literature [2, 3, 10] that
\begin{equation}\sum_{n\ge0}\left(\frac{1}{(q)_{\infty}}-\frac{1}{(q)_n}\right)=\frac{1}{(q)_{\infty}}\sum_{n\ge1}\frac{q^{n}}{1-q^{n}},\label{eq:(1.1)}\end{equation}
which is sometimes noted as a simple example of a``sum of tails" $q$-series. It also has been noted in [10] as equivalent to the result found in [12], that $d(n)$ is equal to the sum of smallest parts of partitions of $n$ into an odd number of distinct parts minus those with an even number of distinct parts. Here we put $\sigma_i(n)=\sum_{d|n}d^{i},$ $d(n)=\sigma_{0}(n),$ and $p(n)$ shall denote the number of unrestricted partitions of $n.$

\begin{theorem}\label{thm:ex1} Let $p(m,n)$ be the number of partitions of $n$ with $m$ parts, and put $p_2(n)=\sum_{m}m^2p(m,n).$ Let $N_2(n)$ be the second Atkin-Garvan moment [5]. Let $f(n,N)$ equal $p(n)$ minus the number of partitions of $n$ where parts are $\le N,$ and $g(n,N)$ the number of divisors $\le N$ of $n.$ Put $S(n)=\sum_{k,N\ge1}f(k,N)g(n-k,N).$ Then,

$$\sum_{n\ge1}S(n)q^n=\sum_{n\ge1}\left(\frac{1}{(q)_{\infty}}-\frac{1}{(q)_n}\right)\sum_{i\ge1}^{n}\frac{q^{i}}{1-q^i}=\sum_{n\ge1}(p_2(n)-np(n)-\frac{1}{2}N_2(n))q^n.$$
\end{theorem}

We mention that the appearance of the finite sum $\sum_{i\ge1}^{n}q^{i}(1-q^i)^{-1}$ suggests a connection with \eqref{eq:(1.1)} through the identity of Van Hamme [7]. The non-negativity of $S(n)$ implies our next result, which includes $M_2(n)$ [5] the second crank moment function.
\begin{theorem}\label{thm:ex2} For each non-negative integer $n,$ $M_2(n)+N_2(n)\le 2p_2(n).$ \end{theorem}

In general it is known through Andrew's [4] $spt(n)=\frac{1}{2}(M_2(n)-N_2(n)),$ that $M_2(n)>N_2(n).$ Hence a natural consequence of this theorem is that $N_2(n)\le p_2(n),$ since $2N_2(n)\le M_2(n)+N_2(n).$ 

To provide a further partition theorem connected to our Theorem \ref{thm:ex1} we will need the following definition.

{\bf Definition 1.} \it We define $P_{i,j}(n_1,n_2,n)$ to be the number of partitions of $n$ into parts $\le n_1+n_2$ wherein each part $\le n_1+n_2$ appears at least once as a part, and $n_1$ appears at least $(1+i)$ times and $n_2$ appears at least $(1+j)$ times. \rm

The generating function is 
\begin{equation}\frac{q}{1-q}\frac{q^2}{1-q^2}\cdots \frac{q^{(1+i)n_1}}{1-q^{n_1}}\cdots \frac{q^{(i+j)n_2}}{1-q^{n_2}}\cdots \frac{q^{n_1+n_2}}{1-q^{n_1+n_2}},\label{eq:(1.2)}\end{equation}
and so we see that $\sum_{i,j,n_1,n_2,n\ge1}(-1)^{n_1+n_2}P_{i,j}(n_1,n_2,n)q^n$ is our $q$-series on the left hand side of (2.5).

\begin{theorem}\label{thm:ex3} Let $D_o(n,N)$ (resp. $D_e(n,N)$) denote the number of partitions of $n$ into an odd (resp. even) number of distinct parts, each $\ge N+1.$ We have,
$$\sum_{i,j,n_1,n_2\ge1}(-1)^{n_1+n_2}P_{i,j}(n_1,n_2,n)=\sum_{k,N\ge1}(D_o(k,N)-D_e(k,N))g(n-k,N).$$
\end{theorem}

We mention this may be viewed as a two-dimensional analogue of the identity [6, 10]
$$\sum_{n\ge1}\frac{(-1)^{n}q^{\binom{n+1}{2}}}{(q)_{n}(1-q^{n})}=\sum_{n\ge0}\left((q^{n+1})_{\infty}-1\right),$$
which, in our notation is also $$\sum_{i,n_1,n\ge1}(-1)^{n_1}P_{i,0}(n_1,0,n)q^n.$$ \par The appearance of the square of the generating function for $d(n)$ creates a bit more of a challenge to work with than some similar identities related to $N_2(n),$ such as the smallest part identity [4]. We were, however, able to apply an interesting formula due to B. Kim [9, eq.(1.6)], $n\ge2,$
\begin{equation} \sum_{k}^{n-1}d(k)d(n-k)=\sigma_1(n)-\sigma_0(n)+2b(2,n),\label{eq:(1.3)}\end{equation}
where $b(m, n)$ denotes the number of partitions of $n$ into $m$ \it different \rm parts. Therefore, we have from Theorem \ref{thm:ex1} and \eqref{eq:(1.3)} that 
\begin{equation}S(n)=np(n)+\sum_{k}p(n-k)(2b(2,k)-d(k))-\frac{1}{2}N_2(n),\label{eq:(1.4)}\end{equation} once we apply the known [6] identity of Euler $np(n)=\sum_{k}\sigma_1(k)p(n-k).$ Hence $S(n)\equiv np(n)-\sum_{k}p(n-k)d(k)-\frac{1}{2}N_2(n)\pmod{2}.$ Another simple inequality may be obtained from (1.4) from the observation that $p(m,n)\ge b(m,n).$

\section{Proof of Theorems}

To prove our theorems we will use some familiar tools that have appeared in previous studies [1, 6, 8, 11]. 
\begin{proof}[Proof of Theorem \ref{thm:ex1}]
From [1], we have that for a $2$-fold Bailey pair $(\alpha_{n_1,n_2},\beta_{n_1,n_2})$ relative to $a_i,$ $i=1,2,$
		\begin{equation}\beta_{n_1, n_2}=\sum_{r_1\ge0}^{n_1}\sum_{r_2\ge0}^{n_2} \frac{\alpha_{r_1,r_2}}{(a_1q;q)_{n_1+r_1} (q;q)_{n_1-r_1}(a_2q;q)_{n_2+r_2} (q;q)_{n_2-r_2}}.\end{equation}
The needed general formula is given by
$$\sum_{n_1\ge0}^{\infty}\sum_{n_2\ge0}^{\infty}(x)_{n_1}(y)_{n_1}(z)_{n_2}(w)_{n_2}(a_1q/xy)^{n_1}(a_2q/zw)^{n_2}\beta_{n_1, n_2}$$
\begin{equation}=\frac{(a_1q/x)_{\infty}(a_1q/y)_{\infty}(a_2q/z)_{\infty}(a_2q/w)_{\infty}}{(a_1q)_{\infty}(a_1q/xy)_{\infty}(a_2q)_{\infty}(a_2q/zw)_{\infty}}\sum_{n_1\ge0}^{\infty}\sum_{n_2\ge0}^{\infty}\frac{(x)_{n_1}(y)_{n_1}(z)_{n_2}(w)_{n_2}(a_1q/xy)^{n_1}(a_2q/zw)^{n_2}\alpha_{n_1, n_2}}{(a_1q/x)_{n_1}(a_1q/y)_{n_1}(a_2q/z)_{n_2}(a_2q/w)_{n_2}}.\label{eq:(2.2)}\end{equation} \rm
Using the Joshi and Vyas [8] $2$-fold Bailey pair with relative to $a_j=1,$ $j=1,2,$ where $\alpha_{0, 0}=1,$

\begin{equation}\alpha_{n_1, n_2}=\begin{cases} (-1)^nq^{n(n-1)/2}(1+q^n),& \text {if } n_1=n_2=n,\\ 0, & \text{otherwise,} \end{cases}\end{equation}
and 

\begin{equation}\beta_{n_1, n_2}=\frac{q^{n_1n_2}}{(q)_{n_1}(q)_{n_2}(q)_{n_1+n_2}}.\end{equation}
 with \eqref{eq:(2.2)} and differentiating \eqref{eq:(2.2)} with respect to $x,$ setting $x=1,$ $y\rightarrow\infty,$ differentiating with respect to $z,$ setting $z=1,$ $w\rightarrow\infty,$
\begin{equation}\sum_{n_1,n_2\ge1}\frac{(-1)^{n_1+n_2}q^{\binom{n_1+n_2+1}{2}}}{(q)_{n_1+n_2}(1-q^{n_1})(1-q^{n_2})}=\left(\sum_{n\ge1}\frac{q^n}{1-q^n}\right)^2+\sum_{n\ge1}\frac{(-1)^n(1+q^n)q^{n(3n+1)/2}}{(1-q^n)^2}.\label{eq:(2.5)}\end{equation}
By Fine's identity [6, pg.13, eq.(12.2), $a=1,$ $t=q^m$], we have 
\begin{equation}\frac{1-q^m}{(q)_m}\sum_{n\ge0}\frac{(q)_n}{(bq)_n}q^{nm}=\sum_{n\ge0}\frac{(b)_n(-1)^nq^{n(n+1)/2+nm}}{(bq)_n(q)_{n+m}}.\label{eq:(2.6)}\end{equation} Differentiating with respect to $b$ and setting $b=1,$ dividing by $1-q^m$ and inverting the desired series we have that the left side of \eqref{eq:(2.5)} is
$$\sum_{n,m\ge1}\frac{(-1)^{m-1}q^{nm+m(m+1)/2}}{(q)_m}\sum_{i\ge1}^{n}\frac{q^i}{1-q^i}=-\sum_{n\ge1}\left((q^{n+1})_{\infty}-1\right)\sum_{i\ge1}^{n}\frac{q^{i}}{1-q^i}.$$ Inserting this into \eqref{eq:(2.5)} and dividing both sides by $(q)_{\infty}$ gives our main identity once we note that
$$\sum_{n\ge1}p_2(n)q^n=\frac{1}{(q)_{\infty}}\left(\sum_{n\ge1}\frac{q^n}{1-q^n}\right)^2+\sum_{n\ge1}np(n)q^n.$$

Now interpreting our $q$-series may be done as follows. $$\frac{1}{(q)_{\infty}}-\frac{1}{(q)_{N}}$$ is the generating function for $f(n,N),$ $p(n)$ minus the number of partitions of $n$ where parts are $\le N.$ The sum $\sum_{1\le i \le N}q^i(1-q^i)^{-1}$ is the generating function for $g(n,N),$ the number of divisors $\le N$ of $n.$ Hence we may write
$$\left(\frac{1}{(q)_{\infty}}-\frac{1}{(q)_{N}}\right)\sum_{1\le i \le N}q^i(1-q^i)^{-1}=\sum_{n\ge1}\left(\sum_{k}f(k,N)g(n-k,N)\right)q^n,$$ and therefore
$\sum_{k,N\ge1}f(k,N)g(n-k,N)$ is the coefficient of $q^n$ on the left side of Theorem 1.1. 
\end{proof}
\begin{proof}[Proof of Theorem \ref{thm:ex2}]
Since it is clear from the generating function that $S(n)\ge 0,$ we have $p_2(n)-np(n)-\frac{1}{2}N_2(n)\ge 0,$ and with $2np(n)=M_2(n)$ [5] we see the result follows. 
\end{proof}

\begin{proof}[Proof of Theorem \ref{thm:ex3}]
From Definition 1, we paraphrase the identity 
$$\sum_{n_1,n_2\ge1}\frac{(-1)^{n_1+n_2}q^{\binom{n_1+n_2+1}{2}}}{(q)_{n_1+n_2}(1-q^{n_1})(1-q^{n_2})}=-\sum_{n\ge1}\left((q^{n+1})_{\infty}-1\right)\sum_{i\ge1}^{n}\frac{q^{i}}{1-q^i},$$
and apply Definition 1.
\end{proof}

\section{Concluding remark}

Here we were able to obtain an interesting sum of tails identity without the use of the methods in [2, 3]. It would appear that one should be able to extend the ideas found in [3], particularly [3, Theorem 4.1, $t=0$], to obtain our Theorem \ref{thm:ex1}. It would also be of interest to see if $S(n)$ has similar divisibility properties as Andrews $spt(n)$ function [4]. Lastly, it would be of interest to see if Theorem \ref{thm:ex2} holds for $k$th moments as well.

1390 Bumps River Rd. \\*
Centerville, MA
02632 \\*
USA \\*
E-mail: alexpatk@hotmail.com
\end{document}